\documentclass[12pt, reqno]{amsart}

\usepackage{amssymb,amsmath,accents}
\usepackage{amsfonts,amsthm,mathrsfs}
\usepackage{cases} 
\usepackage{graphicx} 
\usepackage{subcaption}
\usepackage{diagbox}
\usepackage{url}

\usepackage{txfonts} 

\usepackage{xcolor}

\definecolor{dGREEN}{rgb}{0.0,0.6,0.4}

\numberwithin{equation}{section}
\newtheorem{thm}{Theorem}[section]
\newtheorem{cor}[thm]{Corollary}
\newtheorem{lemma}[thm]{Lemma}
\newtheorem{prop}[thm]{Proposition}

\newtheorem{remark}[thm]{Remark}

\makeatletter
\@namedef{subjclassname@2020}{\textup{2020} Mathematics Subject Classification}
\makeatother

\begin{document}
\title[No formation of a new phase]
{No formation of a new phase for a free boundary problem in combustion theory} 
\dedicatory{In memory of Professor Marek Fila, an esteemed colleague}

\author[K.~Furukawa]{Ken Furukawa}
\address[K.~Furukawa]{Faculty of Science, Academic Assembly, University of Toyama, 3190 Gofuku, Toyama-shi, Toyama 930-8555, Japan}
\email{furukawa@sci.u-toyama.ac.jp}

\author[Y.~Giga]{Yoshikazu Giga}
\address[Y.~Giga]{Graduate School of Mathematical Sciences, The University of Tokyo, 3-8-1 Komaba, Meguro-ku, Tokyo 153-8914, Japan.}
\email{labgiga@ms.u-tokyo.ac.jp.}

\author[N.~Kajiwara]{Naoto Kajiwara}
\address[N.~Kajiwara]{Applied Physics Course Department of Electrical, Electronic and Computer Engineering, Gifu University, 1-1 Yanagido, Gifu 501-1193, Japan.}
\email{kajiwara.naoto.p4@f.gifu-u.ac.jp}


\keywords{non-existence, free boundary problem, self-similar solution, combustion theory, Laplace transform}

\begin{abstract}
We consider a free boundary problem for the heat equation with a given non-negative external heat source, which is often called a parabolic Bernoulli free boundary problem.
 On the free boundary, we impose the zero Dirichlet condition and the fixed normal derivative so that heat  escapes from the boundary.
 In various settings, we show that there exist no solutions when the initial temperature equals the fixed temperature no matter where the initial location of the free boundary is given provided that the external heat source is bounded from above.
 We also note that there is a chance to have a solution when the external temperature is unbounded as time tends to zero by giving a self-similar solution.
\end{abstract}

\maketitle

\section{Introduction} \label{SIn} 

We consider a free boundary problem which often arises in the combustion theory.
 It is of the form
\begin{alignat}{3}
	\partial_t u- \Delta u &= f \quad&&\text{in}
	\quad&& Q_T := \bigcup_{0<t<T} D(t) \times \{t\}, \label{EEq} \\
	u &= 0 \quad&&\text{on}
	\quad&& S_T := \bigcup_{0<t<T} S(t) \times \{t\}, \label{ED} \\
	\partial u/\partial\nu &= -\alpha \quad&&\text{on}
	\quad&& S_T, \label{EN}
\end{alignat}
where $\alpha>0$ is a given constant and $f$ is a given heat source;
 $D(t)$ is a domain in $\mathbb{R}^N$ with boundary $S(t)=\partial D(t)$ and $\partial/\partial\nu$ denotes the exterior normal derivative of $D(t)$ on $S(t)$.
 By a solution of \eqref{EEq}, \eqref{ED} and \eqref{EN}, we mean a classical non-negative solution in the sense that $u=u(x,t)$ ($\ge0$) is $C^2$ in $x$ and $C^1$ in $t$ in $Q_T$ and $u$ together with its spatial gradient $\nabla u$ is continuous up to $S(t)$, i.e., $\nabla^2u$, $\partial_tu\in C(Q_T)$ and $u,\nabla u\in C(Q_T\cup S_T)$; for $D(t)$, we assume that $S_T$ is uniformly $C^{2,1}$ ($C^2$ in space, $C^1$ in time) in $\mathbb{R}^N\times(\varepsilon,T)$ for any $\varepsilon>0$.
 We are interested in its initial value problem with initial data
\begin{equation} \label{EInt}
	\left. u \right|_{t=0} = u_0, \quad
	\left. S(t) \right|_{t=0} = S_0 \ \text{(}=\partial D_0\text{)},
\end{equation}
where $D_0$ is a given domain in $\mathbb{R}^N$.
 By the strong maximum principle \cite{PW}, we notice that $u>0$ in $Q_T$. 
 The unknowns are $u$ and the free boundary $S(t)$.
 In combustion theory, we usually take $f\equiv0$ but the initial data $u_0$ is taken positive in $D_0$ and $u_0=0$ on $\partial D_0$.
 By a solution to the initial value problem \eqref{EEq}, \eqref{ED}, \eqref{EN} with \eqref{EInt}, we mean that $\left(u,S(t)\right)$ is a solution to \eqref{EEq}, \eqref{ED} and \eqref{EN} such that $u$ is bounded and continuous up to $t=0$ on $\overline{D_0}$ and that $S(t)\to S_0$ as $t\downarrow0$ in the uniform $C^2$ topology.
 Note that we do not assume any control of time derivative of $S(t)$ up to $t=0$.
 For example, in one-dimensional setting, the derivative $\dot{s}(t)$ of $s$ may not be integrable in any small interval $[0,\varepsilon)$ if $S(t)$ is written as $x=s(t)$.
 In particular, $S_T$ is allowed to be tangential to the plane $t=0$.
 We are particularly interested in the case $u_0\equiv0$.
 This corresponds to the problem whether or not a new phase (unburnt area) is created.
 Our answer is negative.
 Let us state one-dimensional case.

\begin{thm} \label{T1D}
Assume that $N=1$ and that $f$ is non-negative bounded and measurable on $\mathbb{R}^N\times(0,T)$. 
 If $D_0=(s_0,\infty)$ and $u_0\equiv0$, there exist no solutions to \eqref{EEq}, \eqref{ED}, \eqref{EN} with \eqref{EInt}.
\end{thm}
\begin{cor} \label{C1D}
Under the same hypotheses of Theorem \ref{T1D} concerning $N$ and $f$, if $D_0=(s_{0-},s_{0+})$ with $-\infty<s_{0-}<s_{0+}<\infty$ and $u_0\equiv0$, there exist no solutions to \eqref{EEq}, \eqref{ED}, \eqref{EN} with \eqref{EInt}.
\end{cor}
The proofs are elementary.
 In the case of Theorem \ref{T1D}, $D(t)$ is a half interval, i.e., $D(t)=\left(s(t),\infty\right)$.
 As usual, we change the independent variable by $y=x-s(t)$ and transform the problem \eqref{EEq}, \eqref{ED}, \eqref{EN} in a fixed domain.
 Rewriting $y$ by $x$, its explicit form is
\begin{alignat}{2}
	\partial_t U - \partial_x^2 U - \dot{s} \partial_x U &= f\left(x+s(t),t\right) \label{ETr} \\
	U(0,t) &= 0 \label{EDTr} \\
	\partial_x U(0,t) &= \alpha \label{ENTr}
\end{alignat}
for $U(x,t)=u\left(x+s(t),t\right)$.
 We take the Laplace transform of $U$ in $x$ variable and denote it by $\hat{U}(\lambda,t)$ for $\lambda>0$.
 This $\hat{U}$ solves a linear first order ordinary different equation (ODE) in time and one can calculate 
\begin{equation} \label{EFo} 
	\hat{U}(\lambda,t) = \int_0^t \exp
	\left( -\int_\sigma^t p(\lambda,\tau)\, d\tau \right)
	\left(\hat{g}(\lambda,\sigma) - \alpha \right) d\sigma
\end{equation}
with $p(\lambda,\tau)=-\lambda^2-\dot{s}(\tau)\lambda$, $g(x,t)=f\left(x+s(t),t\right)$.
 Our assumption on $f$ implies that
\[
	\sup_{0<\sigma<T} \hat{g}(\lambda,\sigma) \to 0
	\quad\text{as}\quad \lambda \to \infty.
\]
Thus, for sufficiently large $\lambda$, $\hat{U}(\lambda,t)<0$ which contradicts the non-negativity of $U$.
 We shall give details in Section \ref{SOne}.
\begin{remark} \label{RGr}
From the proof, it is rather clear that the conclusion of Theorem \ref{T1D} is still valid when
\[
	\sup_{x\in\mathbb{R}} f(x,t) e^{-\lambda_0 x}
\]
is bounded from above on $(0,T)$ with some $\lambda_0>0$.
 In other words, exponential growth is allowed.
 Also, $u$ is allowed to be unbounded provided that
\[
	\sup_{x\in\mathbb{R}} u(x,t) e^{-\lambda_0 x}
\]
is bounded from above on $(0,T)$ with some $\lambda_0>0$.
\end{remark}

The proof of Corollary \ref{C1D} is similar once we observe that
\[
	U(x,t) = u \left( x + s_-(t), t \right)
\]
is a subsolution of \eqref{ETr}, \eqref{EDTr}, \eqref{ENTr} for $x>s_-(t)$ when $D(t)=\left(s_-(t), s_+(t)\right)$ by extending $u$ by zero outside $D(t)$.
 We shall give its proof in Section \ref{SOne}.

This idea can be generalized to multi-dimensional setting for special choice of $D_0$.
\begin{thm} \label{TMul}
Assume that $f$ is non-negative, bounded and measurable on $\mathbb{R}^N\times(0,T)$.
 For a unit vector $m\in\mathbb{R}^N$, let $P_m$ be the orthogonal projection of $\mathbb{R}^N$ defined by $P_m x=x-(x\cdot m)m$ for $x\in\mathbb{R}^N$.
 Identify $P_m\mathbb{R}^N$ by $\mathbb{R}^{N-1}$ and $P_m x$ by $x'\in\mathbb{R}^{N-1}$.
 Assume that $s_0:\mathbb{R}^{N-1}\to\mathbb{R}$ is a $C^2$ function with bounded derivative up to second order.
 If $D_0=\left\{ x\in\mathbb{R}^N\bigm|s_0(x')<x\cdot m\right\}$ and $u_0\equiv0$, there exist no solutions to \eqref{EEq}, \eqref{ED}, \eqref{EN} with \eqref{EInt}.
 Here $x'=x-(x\cdot m)m$.
\end{thm}
\begin{cor} \label{CMul}
Assume the same hypotheses of Theorem \ref{TMul} concerning $f$ and $m$.
 Assume that $s_{0\pm}:\mathbb{R}^{N-1}\to\mathbb{R}$ is a $C^2$ function with bounded derivative up to second order with $s_{0-}<s_{0+}$ on $\mathbb{R}^{N-1}$.
 Then there are no solutions to \eqref{EEq}, \eqref{ED}, \eqref{EN} with \eqref{EInt} if
\[
	D_0=\left\{ x\in\mathbb{R}^N\bigm|s_{0-}(x')<x\cdot m<s_{0+}(x') \right\}
	\quad\text{and}\quad u_0 \equiv 0.
\]
\end{cor}

The proof of Theorem \ref{TMul} is similar to that of Theorem \ref{T1D}.
 We may assume that $m=(1,0,\ldots,0)$ by rotation and that
\[
	D(t) = \left\{ (x_1,x') \bigm| x_1 > s(t,x'),\ x'\in\mathbb{R}^{N-1} \right\}.
\]
We set $U(x,t)=u\left(x_1+s(t,x'),x'\right)$ and consider the Laplace transform $\hat{U}(\lambda,x',t)$ in $x_1$-variable.
 In one-dimensional case, $\hat{U}$ solves a linear ODE of the first order in time.
 In multi-dimensional setting, $\hat{U}$ solves a parabolic equation.
 Fortunately, for a sufficiently large $\lambda>0$ (uniformly in $x'$), we are able to prove that the inhomogeneous term of this parabolic equation is negative, so by the maximum principle, we conclude that $\hat{U}<0$, which yields a contradiction,
 The proof for Corollary \ref{CMul} parallels that of Corollary \ref{C1D} once we extend Theorem \ref{T1D} to Theorem \ref{TMul}.

It seems non-trivial to extend these results when $D_0$ is a bounded domain since we cannot take $\lambda$ large uniformly in $x'$-variable.
 It seems also non-trivial to extend our method based on Laplace transforms to radially symmetric case for $N\ge2$.

If $f$ is unbounded close to $t=0$, there may exist a solution for \eqref{EEq}, \eqref{ED}, \eqref{EN} with \eqref{EInt} for $D_0=(0,\infty)$.
 We give an explicit example when $f\sqrt{t}$ is a positive constant.
 In this case, we construct an explicit solution which is self-similar in the sense that
\[
	u(x,t) = \lambda^{-1} u(\lambda x,\lambda^2 t), \quad
	s(t) = \lambda^{-1} s(\lambda^2 t)
\]
for $\lambda>0$, $x>0$.
 This is not difficult by solving a Hermite type differential equation, which is a linear ODE of the second order.

From the view point of combustion theory, the complement of $D(t)$ is burnt zone and $\theta=\theta_c-u$ is a temperature in the unburnt zone, where $\theta_c$ is the temperature on the burning front.
 Our results say that it is impossible to create unburnt zone in already burnt zone by removing heat.

In the case of $f\equiv0$ with $u_0\ge0$, there is a large literature of existence and uniqueness of a solution.
 The paper \cite{GHV} includes the history of the development of the theory.
 The theory is divided into two settings.
 The one-dimensional case \cite{Ve}, \cite{H1} or radially symmetric case \cite{GHV}, \cite{H2} where uniqueness of the solution is guaranteed as well as global existence of a reasonable solution;
 in the latter case $D_0$ is taken either a ball or annulus for $n\ge3$.
 The construction of a solution in \cite{GHV} is based on what is called elliptic-parabolic approach (see \cite{H1} and \cite{H2}) by considering equation
\[
	\partial_t u_+ = \Delta u
\]
with $u_+=\max(u,0)$.
 For general higher dimension setting, a weak solution is constructed by Caffarelli and Vazquez \cite{CV}.
 The main idea in \cite{CV} is to approximate the problem by
\[
	\partial_t u^\varepsilon - \Delta u^\varepsilon = -\beta_\varepsilon (u^\varepsilon)
	\quad\text{for small}\quad \varepsilon > 0, 
\]
where $\beta_\varepsilon(s)=\beta(s/\varepsilon)/\varepsilon$ where $\beta\in C^\infty(\mathbb{R})$ satisfies
\begin{enumerate}
\item[(i)] $\beta>0$ on $(0,1)$ and $\beta=0$ outside $(0,1)$,
\item[(ii)] $\beta$ is increasing in $[0,1/2)$ while $\beta$ is decreasing on $(1/2,1]$,
\item[(iii)] $\int_0^1 \beta\,ds=1/2$.
\end{enumerate}
This is sometimes called Arrhenius approximations.
 For fine properties of the limit of  $u^\varepsilon$ and $\varepsilon\to0$ is studied by Caffarelli and V\'azquez \cite{CV}, Caffarelli, Lederman and Wolanski \cite{CLW1, CLW2} and Weiss \cite{W1, W2}.
 The main difference from one-dimensional or radially symmetric setting is that the solution may not be unique.
 An example is given in \cite{Va} where initial data is in the form of two circular humps.
 Its rigorous proof is given in \cite{PY}.
 There is also a viscosity approach developed by Kim \cite{K}.

Going back to one-dimensional setting, there are several works discussing stability of the traveling solution.
 This problem is studied in \cite{HH1, HH2}, where regularity of solution and interface is also discussed.
 In the setting of Theorem \ref{TMul}, the stability of traveling wave is proved (with $f\equiv0$) is proved in \cite{BHL} in multi-dimensional setting.
 There, the existence of a solution near the traveling wave has been established.

This paper is organized as follows.
 In Section \ref{SOne}, we give proofs of both Theorem \ref{T1D} and Corollary \ref{C1D}.
 In Section \ref{SMul}, we give proofs of Theorem \ref{TMul} and Corollary \ref{CMul}.
 In Section \ref{SSel}, we construct a self-similar solution when $f\sqrt{t}$ is a positive constant with $D_0=(0,\infty)$ and $u_0\equiv0$.
 
\section{One-dimensional problems} \label{SOne}

We shall give a proof of Theorem \ref{T1D} as well as that of Corollary \ref{C1D}.
 It is convenient to recall the Laplace transform $\hat{f}$ of a function $f$ defined on a semi-infinite interval $(0,\infty)$.
 Its explicit form is given as 
\[
	\hat{f}(\lambda) := \int_0^\infty e^{-\lambda x} f(x)\, dx.
\]
This is finite for $\lambda>0$ if $f$ is bounded and measurable on $(0,\infty)$, i.e., $f\in L^\infty(0,\infty)$.
 If, moreover $f$ is Lipschitz so that its distributional derivative $\partial_x f\in L^\infty(0,\infty)$, then integration by parts yields
\begin{align*}
	\widehat{\partial_x f} (\lambda) 
	&= \int_0^\infty e^{-\lambda x} (\partial_x f) (x)\, dx
	\equiv e^{-\lambda x} \left.f\right|_0^\infty 
	- \int_0^\infty \left( \partial_x(e^{-\lambda x}) \right) f(x)\, dx \\
	&= -f(0) + \lambda \hat{f}(\lambda).
\end{align*}
We repeat this argument and obtain the following well-known formulas.
\begin{prop} \label{PDL}
Assume that $\partial_x^j f\in L^\infty(0,\infty)$ for $0\leq j\leq k$ with $k\geq1$.
 Then
\[
	\widehat{\partial_x^k f} (\lambda) = \lambda^k \hat{f}(\lambda) - \sum_{j=1}^k \lambda^{k-j} (\partial_x^{j-1}f) (0).
\]
The formula still valid if $\partial_x^k f$ is a finite Radon measure instead of $L^\infty$ function.
\end{prop}
\begin{proof}[Proof of Theorem \ref{T1D}]
Since we assume $u$ is bounded, $\partial_xu$ and $\partial_x^2u$ are bounded for $t\geq\delta$ for any $\delta>0$ by interior estimates;
 see e.g.\ \cite{F}, \cite{LSU}.
 Thus the derivatives $\partial_xU(x,t)$, $\partial_x^2U(x,t)$ is bounded on $[0,\infty)$ as a function of $x$ for each $t>0$, where $U$ is the transformed function.
 By \eqref{EDTr} and \eqref{ENTr}, Proposition \ref{PDL} yields
\begin{align*}
	\widehat{\partial_x^2 U} (\lambda,t) 
	&= \lambda^2 \hat{U}(\lambda,t) - \lambda U(0,t) - \partial_x U(0,t) \\
	&= \lambda^2 \hat{U}(\lambda,t) - \alpha, \\
	\widehat{\partial_x U} (\lambda,t) 
	&= \lambda \hat{U}(\lambda,t),
\end{align*}
where the Laplace transform is taken as a function of $x\in(0,\infty)$.
 Taking the Laplace transform of both side of \eqref{ETr}, we now obtain
\begin{equation} \label{ELap1}
	\widehat{\partial_t U} - \lambda^2 \hat{U} + \alpha
	- \dot{s} \lambda \hat{U} = \hat{g}
\end{equation}
with $g(x,t)=f\left(x+s(t),t\right)$.
 Since we assume that $u$ and $s$ are continuous up to $t=0$, $U(x,t)\to u_0(x)\equiv0$ locally uniformly on $[0,\infty)$.
 Since $U$ is bounded up to $t=0$, this implies that
\[
	\lim_{t\downarrow0} \hat{U} (\lambda,t) = 0
	\quad\text{for}\quad \lambda > 0.
\]
Indeed,
\begin{align*}
	\varlimsup_{t\downarrow0} \left|\hat{U}(\lambda,t)\right|	
	&\leq \varlimsup_{t\downarrow0} \int_0^L e^{-\lambda x} \left|U(x,t)\right| dx
	+ \int_L^\infty e^{-\lambda x} \lVert U \rVert_{L^\infty(\mathbb{R}_+^2)}\, dt \\
	&\leq \lVert U \rVert_\infty e^{-\lambda L}/\lambda \to 0
	\quad\text{as}\quad L \to \infty.
\end{align*}
Since $\partial_tU$ is bounded as a function of $x$ for $t>0$, we easily see that
\[
	\widehat{\partial_t U} = \partial_t \hat{U}.
\]
By \eqref{ELap1}, we now observe that
\begin{equation}
\begin{split} \label{ELap2}
	\left \{
	\begin{array}{rl}
	\partial_t \hat{U} + p(\lambda,t) \hat{U} &\hspace{-0.5em}= \hat{g}-\alpha,
	\quad p(\lambda,t) = -\lambda^2 - \dot{s} (t) \lambda \\
	\hat{U} (\lambda,0) &\hspace{-0.5em}= 0.
	\end{array}
	\right.
\end{split}
\end{equation}
This is a linear first order ordinary equation and it is easy to solve.
 Under the zero initial condition, the explicit form of the solution is
\[
	\hat{U}(\lambda,t)
	= \int_0^t \exp \left( -\int_\sigma^t p(\lambda,\tau)\, d\tau \right)
	\left(\hat{g}(\lambda,\sigma) - \alpha \right) d\sigma.
\]
Here, we only invoke that
\[
	s \in C[0,T) \cap C^1 \left((0,T)\right);
\]
we do not assume $\int_0^T|\dot{s}|\,d\tau<\infty$.
 If $f$ is bounded, then
\[
	\hat{g}(\lambda,\sigma)
	\leq \lVert f \rVert_\infty
	\int_0^\infty e^{-\lambda x}\, dx
	= \lVert f \rVert_\infty / \lambda.
\]
Thus
\[
	\hat{g}(\lambda,\sigma) - \alpha
	\leq -\alpha/2 \quad\text{for}\quad \sigma\in (0,T), \quad
	\lambda \geq 2\lVert f \rVert_\infty/\alpha.
\]
This implies that
\[
	\hat{U}(\lambda,t) < 0
\]
if $\lambda$ is large, say $\lambda\geq2\lVert f \rVert_\infty/\alpha$.
 This contradicts the non-negativity of $u$.
\end{proof}
We shall prove Corollary \ref{C1D} by observing that $\bar{u}$ is a subsolution in
\[
	\tilde{Q}_T = \bigcup_{0<t<T} \left( s_-(t),\infty \right) \times \{t\},
\]
where $\bar{u}$ denotes the zero extension of $u$ to $x>s_+(t)$.
\begin{proof}[Proof of Corollary \ref{C1D}]
We write
\[
	D(t) = \left( s_-(t), s_+(t) \right).
\]
 Let $\bar{f}$ denote the zero extension of $f$ for $x>s_+(t)$ for a function $f$ defined for $x>s(t):=s_-(t)$.
 Since $\bar{u}=0$ for $x>s_+(t)$ and $u_x=-\alpha$ at $s_+(t)$, we first note that
\[
	\partial_x^2 \bar{u} = \overline{\partial_x^2u}
	+ \alpha\delta \left(x-s_+(t) \right)
\]
for $x>s(t)$.
 For $\partial_x u$ and $\partial_t u$, we observe that $\partial_x\bar{u}=\overline{\partial_xu}$, $\partial_t\bar{u}=\overline{\partial_tu}$ for $x>s(t)$.
 We set
\[
	U(x,t) = \bar{u} \left( x + s(t),t \right)
\]
as before.
 The equation \eqref{EEq} now yields
\[
	 \partial_t U - \partial_x^2 U - \dot{s}\partial_x U
	 = \bar{f} \left( x + s(t),t \right) - \alpha\delta \left( x - s_+(t) + s(t) \right),
\]
instead of \eqref{ETr}.
 We proceed as before and obtain that
\[
	\hat{U}(\lambda,t) \leq
	\int_0^t \exp \left( -\int_\sigma^t p(\lambda, \tau) \, d\tau \right)
	\left( \hat{\bar{g}}(\lambda,\sigma) - \alpha \right) d\sigma
\]
with $\bar{g}(x,t)=\bar{f}\left(x+ s(t),t\right)$.
 We conclude that the right-hand side is negative for sufficiently large $\lambda$.
 This implies that $\hat{U}(\lambda,t)<0$ for sufficiently large $\lambda$ and this contradicts the non-negativity of $u$.
 Thus, there is no solution.
\end{proof}

\section{Multi-dimensional problems} \label{SMul}

We extend the idea for one-dimensional setting to multi-dimensional setting.
 We begin by proving Theorem \ref{TMul}.
\begin{proof}[Proof of Theorem \ref{TMul}]
We may take $m=(1,0,\ldots,0)$ by notation.
 We set
\[
	U(x,t) = u(x_1 + s(t,x'), x', t), 
	\quad x = (x_1, x').
\]
We notice that
\[
	\nu=\frac{1}{\sqrt{1+|\nabla_{x'}s|^2}} (-1, \nabla_{x'}s).
\]
Since $\partial u/\partial \nu = \nabla u \cdot \nu = -\alpha$ on $S(t)$ and $u=0$ on $S(t)$ by \eqref{EN} and \eqref{ED}, we observe that $\nabla u=-\alpha\nu$ on $S(t)$. 
 This implies that $\partial u/\partial x_1=\alpha/\sqrt{1+|\nabla_{x'} s|^2}$.
 Then the problem \eqref{EEq}, \eqref{ED}, \eqref{EN} is transformed as
\begin{align} 
	\partial_t U - \partial_{x_1}^2 U 
	&- \dot{s}\partial_{x_1} U - \sum_{i=2}^N \left( \partial_{x_i} - (\partial_{x_i} s) \partial_{x_1} \right) ^2 U \notag \\
	&= f \left( x_1 + s(t,x'), x', t \right), \quad
	x_1 > 0,\ t \in (0,T),\ x' \in \mathbb{R}^{N-1} \label{EMTr} \\
	U(0, x', t) &= 0, \quad
	t \in (0,T), \ x' \in \mathbb{R}^{N-1} \label{EMDTr} \\
	\partial_{x_1} U(0, x', t) &= \alpha\Bigm/\sqrt{1+|\nabla_{x'} s|^2}, \quad
	t \in (0,T), \ x' \in \mathbb{R}^{N-1}, \label{EMNTr}\end{align}
where $\nabla_{x'}=(\partial_{x_2},\ldots,\partial_{x_N})$ denotes the gradient in $x'$.

We shall take the Laplace transform of \eqref{EMTr} (with respect to $x_1$ variable) of both sides.
 Let $\hat{U}(\lambda,x',t)$ denote
\[
	\hat{U} (\lambda, x', t)
	= \int_0^\infty e^{-\lambda x_1} U (x_1, x', t)\, dx_1.
\]
In a similar way to derive \eqref{ELap2}, the equations \eqref{EMTr}, \eqref{EMDTr}, \eqref{EMNTr} yield
\begin{equation}
\begin{split} \label{EMLap}
	\left \{
	\begin{array}{l}
	\partial_t \hat{U} + \left(p(\lambda, x', t) - L(\lambda, x', t) \right) \hat{U} = \hat{g} - \sqrt{1+|\nabla_{x'} s|^2} \alpha \\
	\hat{U} (\lambda, x', 0) = 0
	\end{array}
	\right.
\end{split}
\end{equation}
with $p(\lambda, x', t) :=-\lambda^2-\dot{s}(t,x')\lambda$,
\begin{align*}
	L (\lambda, x', t)
	& := \sum_{i=2}^N \left( \partial_{x_i} - (\partial_{x_i} s(t, x')) \lambda \right) ^2,\\
	g (x_1, x', t)
	& := f \left( x_1 + s(t,x'), x', t \right).
\end{align*}
In the above calculation, we used the formula that the Laplace transform of $\sum_{i=2}^N \left( \partial_{x_i} - (\partial_{x_i} s) \partial_{x_1} \right)^2 U$ is $L (\lambda, x', t)\hat{U} - \alpha \vert \nabla_{x'} s \vert^2 / \sqrt{1 + \vert \nabla_{x'} s \vert^2}$.
 The solution $\hat{U}$ of \eqref{EMLap} can be written at least formally
\[
	\hat{U} (\lambda,x',t) = \int_0^t T(t,\sigma)
	\left( \hat{g} (\lambda,x',\sigma) - \sqrt{1+|(\nabla_{x'} s)(\sigma, x')|^2} \alpha \right) d\sigma,
\]
where $T(t,\sigma)$ is a propagator of $L-p$, i.e., $w(t)=T(t,\sigma)v_0$ is a solution of $dw/dt=(L-p)w$ for $t>\sigma$ with $w(\sigma)=v_0$.
 This can be justified if $g$ is in $C^\theta\bigl(0,T;BUC(\mathbb{R}^{N-1})\bigr)$ and $s$ is in $C^{1+\theta}\bigl(0,T;BUC(\mathbb{R}^{N-1})\bigr)\cap C^\theta\bigl(0,T;BUC^2(\mathbb{R}^{N-1})\bigr)$ for some $\theta\in(0,1]$, where $C^{k+\theta}$ denotes the space of all $\theta$-H\"older continuous functions with their derivatives up to $k$th order and $BUC^k$ denotes the space of all bounded uniformly continuous functions with their derivatives up to $k$th order and $BUC=BUC^0$; 
 see e.g.\ \cite{L} or \cite{T}, but we do not use this representation. 
 We note that 
\begin{equation*}
	L - p = \Delta' - 2\lambda \sum_{i=2}^N \partial_{x_i} s \partial_{x_i} + (1+\sum_{i=2}^N|\partial_{x_i} s|^2) \lambda^2 + (\dot{s} - \Delta' s) \lambda, 
\end{equation*}
where $\Delta' = \sum_{i=2}^N \partial_{x_i}^2$ denotes Laplacian with respect to $x'$. 

As in the one-dimensional setting,
\[
	\hat{g} (\lambda,x',\sigma) - \sqrt{1+|\nabla_{x'}s|^2} \alpha \leq - \sqrt{1+|\nabla_{x'}s|^2} \alpha/2
\]
for $\sigma\in(0,T)$, $x'\in\mathbb{R}^{N-1}$ provided that
\[
	 \lambda \ge \lVert f \rVert_\infty /(2\alpha).
\]
By this choice of $\lambda$, the right-hand side of \eqref{EMLap} is negative in $\mathbb{R}^{N-1}\times(0,T)$.
 We notice that $\hat{U}$ is bounded and continuous on $\mathbb{R}^{N-1}\times[0,T)$ with $\left.\hat{U}\right|_{t=0}=0$.
 Since $\hat{U}$ solves a parabolic linear equation \eqref{EMLap} with bounded continuous coefficients on $\mathbb{R}^{N-1}\times(\delta,T)$ for any $\lambda>0$ with some control of coefficients as $t\to0$, by applying the maximum principle, we conclude that $\hat{U}(\lambda,x',t)\leq0$ for all $x'\in\mathbb{R}^{N-1}\times(0,T)$ for large $\lambda>0$; for a detailed statement for the maximum principle we use, see Lemma \ref{LMax} and its Remark \ref{RIM} below.
 Since $\hat{U}(\lambda,x',t)\geq0$, this implies $u\equiv0$.
 Thus there is no solution of our problem.
\end{proof}
\begin{lemma} \label{LMax}
Let $w\in C\left(\mathbb{R}^{N-1}\times[0,T)\right)$ be a bounded solution of
\[
	\partial_t w - Pw = g \le 0	\quad\text{in}\quad 
	\mathbb{R}^{N-1} \times (0,T)
\]
with $\left.w\right|_{t=0}=0$.
 Assume that $P$ is an operator of the form
\[
	P = \sum_{1\leq i,j\leq N-1} a_{ij} (z,t) \frac{\partial^2}{\partial z_i \partial z_j}
	+ \sum_{\ell=1}^{N-1} b_\ell (z,t) \frac{\partial}{\partial z_\ell} + c(z,t), \
	z \in \mathbb{R}^{N-1}, \ t \in (0,T).
\]
Assume that $a_{ij},b_\ell, c=c_1+c_2, c_1, g \in C\left(\mathbb{R}^{N-1}\times(0,T)\right)$ and $a_{ij},b_\ell,g \in L^\infty\left(\mathbb{R}^{N-1}\times(0,T)\right)$.
 Assume that $\left\lVert (c_2)_+(\cdot, t) \right\rVert_{L^\infty(\mathbb{R}^{N-1})}<\infty$ ($a_+=\max(a,0)$) for $t\in(0,T)$ and
\[
	C_1(z,t) := \lim_{\sigma\downarrow0} \int_\sigma^t c_1(z,\tau)\, d\tau
	\eqqcolon \int_{0\downarrow}^t c_1(z,\tau)\, d\tau
\]
exists for all $(z,t)\in\mathbb{R}^{N-1}\times(0,T)$ and $\partial_{z_i}\partial_{z_j}C_1,\partial_{z_i}C_1,C_1\in L^\infty\left(\mathbb{R}^{N-1}\times(0,T)\right)$ for $1\le i,j\le N-1$. 
 Assume that
\begin{align*}
	\sum_{1\leq i,j\leq N-1} a_{ij} (z,t) \xi_i \xi_j \geq 0
	\quad&\text{for}\quad \xi \in \mathbb{R}^{N-1},
	\ z \in \mathbb{R}^{N-1} \\
	&\text{with}\quad \xi = (\xi_1,\ldots,\xi_{N-1}).
\end{align*}
Then $w\leq0$ on $\mathbb{R}^{N-1}\times[0,T)$.
\end{lemma}
\begin{proof}[Proof of Lemma \ref{LMax}]
This version of the (weak) maximum principle is, in principle, rather standard;
 see e.g.\ \cite{PW} and \cite{PS}.
 We give a proof for the reader's convenience.
\begin{enumerate}
\item[1.]
 We may assume that
\[
	\lim_{|z|\to\infty} \sup_{0\leq t < T} w(z,t) = 0.
\]
We take a positive $\rho\in C^\infty(\mathbb{R}^{N-1})$ with $\lim_{|z|\to\infty}\rho(z)=0$ and set
\[
	\tilde{w}(z,t) = \rho(z) w(z,t).
\]
Then $\tilde{w}$ solves
\[
	\partial_t \tilde{w} - Q\tilde{w} = \rho g
\]
with
\begin{align*}
	&Q = \sum_{1\leq i,j\leq N-1} a_{ij} \frac{\partial^2}{\partial z_i \partial z_j}
	+ \sum_{\ell=1}^{N-1} \tilde{b}_\ell \frac{\partial}{\partial z_\ell} + \tilde{c}, \\
	&\tilde{b}_\ell= b_\ell - \sum_{i=1}^{N-1} (a_{\ell i} + a_{i\ell}) \rho_{z_i} \bigm/ \rho, \\
	&\tilde{c}_1 = c_1, \quad
	\tilde{c}_2 = c_2	- \sum_{\ell=1}^{N-1} b_\ell \rho_{z_\ell}/\rho
	- \sum_{1\leq i,j\leq N-1} a_{ij} ( \rho_{z_i z_j} \rho
	 - 2 \rho_{z_i}\rho_{z_j}) \bigm/\rho^2, \\
	 &\tilde{c} = \tilde{c}_1 + \tilde{c}_2,
\end{align*}
since
\begin{align*}
	 \rho w_{z_i}
	&= \tilde{w}_{z_i} - \frac{\rho_{z_i}}{\rho} \tilde{w},\\
	 \rho w_{z_i z_j} &= \tilde{w}_{z_i z_j} - \rho_{z_j} w_{z_i} - \rho_{z_i} w_{z_j} - \frac{\rho_{z_i z_j}}{\rho} \tilde{w} \\
	&= \tilde{w}_{z_i z_j} - \frac{\rho_{z_j}}{\rho} \left( \tilde{w}_{z_i} -  \frac{\rho_{z_i}}{\rho} \tilde{w} \right)
	- \frac{\rho_{z_i}}{\rho} \left( \tilde{w}_{z_j} -  \frac{\rho_{z_j}}{\rho} \tilde{w} \right)
	- \frac{\rho_{z_i z_j}}{\rho} \tilde{w},
\end{align*}
where subscript $z_i$ denotes derivative with respect to $z_i$, i.e., $\rho_{z_i}=\partial\rho/\partial z_i$, $\rho_{z_i z_j}=\partial^2\rho\bigm/(\partial z_i \partial z_j)$.
 If
\[
	\rho_{z_i} / \rho, \quad
	\rho_{z_i z_j} / \rho, \quad
	\rho_{ z_i} \rho_{z_j} /\rho^2
\]
is bounded, then $Q$ satisfies all assumptions of $P$.
 Such $\rho$ with
\[
	\lim_{|z|\to\infty}\rho(z)=0
\]
is easily found.
 Indeed, if we take $\rho\in C^\infty(\mathbb{R}^{N-1})$ with $\rho(z)=1/|z|$ for $|z|\geq1$, then $\rho$ fulfills all desired properties.
 Thus
\[
	\lim_{|z|\to\infty} \sup_{0\leq t\leq T} \tilde{w} (z,t) = 0.
\]
(We may relax the assumption for $w$ allowing polynomial growth as $|z|\to\infty$ by taking $\rho(z)=e^{-|z|}$ for $|z|\geq1$.
 However, if $w$ allows exponential growth, there is no desired $\rho$.
 Indeed, the statement is false in general for such growing $w$.)
 
\item[2.]
 We may assume that $\tilde{c}_1=c_1=0$. 
 Indeed, we set $q(z,t)=\exp\left(-C_1(z,t)\right)$. 
 We define $w_1=q\tilde{w}$.
 Since we assume that $C_1$ is bounded on $\mathbb{R}^{N-1}\times(0,T)$, $w_1$ is a bounded solution satisfying
\[
	\lim_{|z|\to\infty} \sup_{0\le t\le T} w_1(z,t) = 0
\]
of
\[
	\partial_t w_1 - Q_* w_1 - c_2^* w_1 = q \rho g
\]
with $Q_*$ and $c_2^*$ of the form 
\begin{align*}
	Q _*&= \sum_{1\leq i,j\leq N-1} a_{ij} \frac{\partial^2}{\partial z_i \partial z_j}
	+ \sum_{\ell=1}^{N-1} \tilde{\tilde{b}}_\ell \frac{\partial}{\partial z_\ell}, \\
	\tilde{\tilde{b}}_\ell&= \tilde{b}_\ell - \sum_{i=1}^{N-1} (a_{\ell i} + a_{i\ell}) q_{z_i} \bigm/ q \\
	&=\tilde{b}_\ell + \sum_{i=1}^{N-1} (a_{\ell i} + a_{i\ell}) (C_1)_{z_i}, \\
	c_2^* &=  \tilde{c}_2	- \sum_{\ell=1}^{N-1} \tilde{b}_\ell q_{z_\ell}/q
	- \sum_{1\leq i,j\leq N-1} a_{ij} (q_{z_i z_j}q
	 - 2 q_{z_i}q_{z_j}) \bigm/q^2\\
	 &=\tilde{c}_2	+ \sum_{\ell=1}^{N-1} \tilde{b}_\ell (C_1)_{z_\ell}
	+ \sum_{1\leq i,j\leq N-1} a_{ij} ((C_1)_{z_i z_j}
	 + (C_1)_{z_i}(C_1)_{z_j}).
\end{align*}
Since we assume that $(C_1)_{z_i}$, $(C_1)_{z_i z_j}$ are bounded in $\mathbb{R}^{N-1}\times(0,T)$, so are $\tilde{\tilde{b}}_\ell$ and $(c_2^*)_+$ for $\ell=1,\ldots,N-1$.
 Thus, the operator $Q_*+c_2^*$ satisfies the all assumptions of $P$ with $c_1=0$.
\item[3.] We may assume that $c_2^*\le-1$.
 Indeed, we set $\mu(t)=t\left(\sup_{\tau\in(0,T)}\sup_z c_2^*(z,\tau)+1 \right)$ and $\tilde{q}=\exp\left(-\mu(t)\right)$.
 We define $\tilde{\tilde{w}}=\tilde{q}w_1$.
 Then $\tilde{\tilde{w}}$ solves
\[
	\partial_t \tilde{\tilde{w}} - Q_* \tilde{\tilde{w}}
	- \tilde{\tilde{c}}_2 \tilde{\tilde{w}}
	= \tilde{q} q \rho g
\]
with $\tilde{\tilde{c}}_2=c_2^* - \sup_{\tau\in(0,T)}\sup_z c_2^*(z,\tau)-1\le-1$. 
\item[4.]
 Suppose that $\tilde{\tilde{w}}$ takes a positive value at $(z_0,t_0)$,
$z_0\in\mathbb{R}^{N-1}$, $t_0\in(0,T)$.
 By Step 1, $\tilde{\tilde{w}}$ must take a positive maximum on $\mathbb{R}^{N-1}\times(0,t_0]$.
 Let $(z_*,t_*)$ be a maximum point, i.e.,
\[
	\tilde{\tilde{w}}(z_*,t_*) = \max_{\mathbb{R}^{N-1}\times[0,t_0]} \tilde{\tilde{w}} > 0.
\]
Since $\left.w\right|_{t=0}=0$, $t_*>0$.
 At the maximum point, we see
\begin{align*}
	&\partial_t \tilde{\tilde{w}} (z_*,t_*) \ge 0, \\
	&\nabla \tilde{\tilde{w}} (z_*,t_*) = 0, \\
	&\sum_{1\le i,j \le N-1} \tilde{\tilde{w}}_{z_i z_j} (z_*,t_*) \xi_i \xi_j \le 0
	\quad\text{for}\quad \xi \in \mathbb{R}^{N-1}.
\end{align*}
Since
\[
	0 \ge \left(\partial_t \tilde{\tilde{w}} - Q_* \tilde{\tilde{w}} - \tilde{\tilde{c}}_2 \tilde{\tilde{w}} \right) (z_*,t_*)
	\ge - \left(\tilde{\tilde{c}}_2 \tilde{\tilde{w}} \right) (z_*,t_*),
\]
we observe that 
\[
	\tilde{\tilde{c}}_2 (z_*,t_*) \ge 0,
\]
which yields a contradiction.
 Thus, $\tilde{\tilde{w}}\le0$ on $\mathbb{R}^{N-1}\times(0,T)$.
 This implies that $w\le0$ on $\mathbb{R}^{N-1}\times(0,T)$.
\end{enumerate}
\end{proof}
\begin{remark} \label{RIM}
To prove Theorem \ref{TMul}, we apply Lemma \ref{LMax} with
\begin{align*}
	& a_{ij} = \delta_{ij}, \quad b_\ell = -2\lambda\partial_{x_{\ell-1}} s, \quad
	c_1 = \dot{s}\lambda, \quad
	c_2 = (-\Delta's) \lambda + \left( 1 + \sum_{i=2}^N |\partial_{x_i}s|^2 \right) \lambda, \\
	& g = \hat{g} - \alpha \Bigm/ \sqrt{1+|\nabla_{x'}s|^2}
	\quad\text{with}\quad
	z = (x_2, \ldots, x_N).
\end{align*}
We do not assume that $|\dot{s}|$ is integrable in $t$ on $(0,T)$.
 However, since we assume that $s$ is continuous up to $t=0$ and that $\dot{s}$ is bounded on $(\delta,T)\times\mathbb{R}^{N-1}$ for any $\delta>0$, the improper integral
\[
	s(t,z) = \int_{0\downarrow}^t \dot{s}(\tau,z)\, d\tau, \quad
	z = (z_2, \ldots, x_N)
\]
exists for all $z\in\mathbb{R}^{N-1}$.
 Since we assume that $S(t)\to S_0$ as $t\to0$ uniformly in $C^2$ topology, $s$, $\partial_{z_i}s$, $\partial_{z_i}\partial_{z_j}s$ are bounded in $\mathbb{R}^{N-1}\times(0,T)$.
 Thus, the assumption on $c_1$ is fulfilled.
 (We do not use the continuity of $\partial_{z_i}s$, $\partial_{z_i}\partial_{z_j}s$ up to $t=0$.)
 It is easy to see that all other assumptions on Lemma \ref{LMax} are fulfilled.
 Thus, we are able to apply Lemma \ref{LMax} to conclude for sufficiently large $\lambda$,
\[
	\hat{U} (\lambda,x',t) \le 0 \quad\text{for all}\quad
	x' \in \mathbb{R}^{N-1} \times (0,T).
\]
\end{remark}
Corollary \ref{CMul} can be proved in a similar way to prove Corollary \ref{C1D} by noting that $U(x_1,x',t)$ is a subsolution of \eqref{EMTr}.

\section{Self-similar solutions} \label{SSel}

If the external source $f$ is unbounded as $t\downarrow0$, there may exist a solution to \eqref{EEq}, \eqref{ED}, \eqref{EN} with \eqref{EInt}.
 We shall give an example of a self-similar solution.
\begin{thm}[Existence of a self-similar solution] \label{TES}
Assume $f=h/\sqrt{t}$ and $N=1$ with a constant $h>0$.
 Assume that $D_0=(0,\infty)$.
 Then there exists a unique bounded solution to \eqref{EEq}, \eqref{ED}, \eqref{EN} with \eqref{EInt} satisfying $D(t)=\left(s(t),\infty\right)$ with $s(t)=\sigma\sqrt{t}$ with a constant $\sigma\in\mathbb{R}$.
 The function $h\mapsto\sigma(h)$ is strictly decreasing smooth function.
 Moreover, $\sigma\left(\alpha/\sqrt{\pi}\right)=0$ and $\sigma<0$ (resp.\ $\sigma>0$) if and only if $h>\alpha/\sqrt{\pi}$ (resp.\ $h<\alpha/\sqrt{\pi}$).
\end{thm}
In the rest of this section, we shall prove Theorem \ref{TES}.
 We consider \eqref{EEq} with $f=h/\sqrt{t}$ and
\begin{align*}
	u(x,t) &= \sqrt{t} w \left( x/\sqrt{t} \right), \\
	s(t) &= \sigma \sqrt{t}.
\end{align*}
Since
\begin{align*}
	\partial_t u &= \frac{1}{2\sqrt{t}} w \left( x/\sqrt{t} \right)
	+ \sqrt{t} \left(\partial_t\left( \frac{x}{\sqrt{t}} \right)\right) w' \left( x/\sqrt{t} \right) \\
	&= \frac{1}{2\sqrt{t}} w(y) - \frac{1}{2\sqrt{t}} yw'(y)
	\quad\text{with}\quad y = x/\sqrt{t}, \\
	\partial_x^2 u &= \frac{1}{\sqrt{t}} w''(y),
\end{align*}
our system \eqref{EEq}, \eqref{ED}, \eqref{EN} is now of the form
\begin{align}
	\frac12 w(y) - \frac{y}{2} w'(y) - w''(y)
	&= h, \quad y > \sigma \label{EHH} \\
	w(\sigma) &= 0 \label{EHD} \\
	w'(\sigma) &= \alpha. \label{EHN}
\end{align}
The first ODE is of Hermite type and we are able to solve it explicitly. 
\begin{prop} \label{PGe}
Let $W$ be a solution of
\begin{equation} \label{EInh}
	\frac12 w - \frac{y}{2} w'(y) - w''(y) = h.
\end{equation}
Then, it can be written as
\[
	W(y) = -c_1 W_1(y) + c_0 W_0(y) + 2h
\]
with some constant $c_1,c_0\in\mathbb{R}$, where
\[
	W_0(y) = y, \quad
	W_1(y) = y \int_y^\infty e^{-\zeta^2/4} \frac{1}{\zeta^2}\, d\zeta.
\]
\end{prop}
\begin{remark} \label{RRev}
Apparently, the function $W_1$ is only defined for $y>0$ as an analytic function.
 As we shall see later, $W_1(y)$ can be extended analytically for all $y\in\mathbb{R}$;
 see formula \eqref{EAF}.
\end{remark}
\begin{proof}
Clearly, $W_0$ is a solution of the homogeneous equation
\[
	\frac12 w - \frac{y}{2} w' - w'' = 0.
\]
Since $v=w/W_0$ solves $(\log v')'=-\left(\frac{y}{2}+\frac{2}{y}\right)$, we see $W_1$ is also a solution to the homogeneous equation so $\{W_0,W_1\}$ forms a fundamental system of the homogeneous equation.
 Since $w\equiv2h$ is a special solution to the inhomogeneous \eqref{EInh} equation, we obtain the desired representation formula.
\end{proof}

To prove Theorem \ref{TES}, it suffices to prove the next lemma.
\begin{lemma} \label{LUB}
There exists a unique $\sigma$ such that the system \eqref{EHH}, \eqref{EHD}, \eqref{EHN} admits a bounded solution.
 Moreover, $h\longmapsto\sigma(h)$ is $C^\infty$ and $\sigma'<0$.
 Furthermore, $\sigma(h)>0$ (resp.\ $\sigma(h)<0$) for $0<h<\alpha/\sqrt{\pi}$ (resp. $h>\alpha/\sqrt{\pi}$) and $\sigma\left(\alpha/\sqrt{\pi}\right)=0$.
\end{lemma}

We first study the behavior of
\[
	W_1(y) = y \int_y^\infty e^{-\zeta^2/4} \frac{d\zeta}{\zeta^2}.
\]
Differentiating $W_1$ yields
\[
	W'_1(y) = -y e^{-y^2/4} \frac{1}{y^2}
	+ \int_y^\infty e^{-\zeta^2/4} \frac{d\zeta}{\zeta^2}.
\]
Integrating by parts yields
\begin{align}
	&\int_y^\infty e^{-\zeta^2/4} \frac{d\zeta}{\zeta^2}
	= \int_y^\infty e^{-\zeta^2/4} \frac{d}{d\zeta} \left( -\frac{1}{\zeta} \right) d\zeta \notag \\
	=-&\int_y^\infty \frac{d}{d\zeta} \left( e^{-\zeta^2/4} \right) \left( -\frac{1}{\zeta} \right) d\zeta
	+ \left[ e^{-\zeta^2/4} \left( -\frac{1}{\zeta} \right) \right]_y^\infty \notag \\
	=-&\int_y^\infty e^{-\zeta^2/4} \frac{\zeta}{2} \frac{1}{\zeta} d\zeta
	+ e^{-y^2/4} \frac{1}{y}. \label{EIDN}
\end{align}
Thus,
\begin{align*}
	W'_1 &= -y e^{-y^2/4} \frac{1}{y^2}
	+ e^{-y^2/4} \frac{1}{y} - \frac12 \int_y^\infty e^{-\zeta^2/4}\, d\zeta. \\
	&= -\frac12 \int_y^\infty e^{-\zeta^2/4}\, d\zeta 
	\quad\text{for all}\quad y \in \mathbb{R}.
\end{align*}
In particular, $W_1$ is strictly decreasing and
\[
	W''_1(y) = \frac12 e^{-y^2/4} > 0
	\quad\text{for}\quad y \in \mathbb{R}.
\]
In particular, $W_1$ is convex.
 Since
\[
	W_1(y) \leq y \int_y^\infty e^{-\zeta^2/4}\, d\zeta \frac{1}{y^2}
	\leq \frac1y \int_y^\infty e^{-\zeta^2/4}\, d\zeta,
\]
we see that
\[
	\lim_{y\to\infty} y W_1(y) = 0.
\]
Its profile is as in Figure \ref{FW1}.
\begin{figure}[htb]
\centering
\includegraphics[width=50mm]{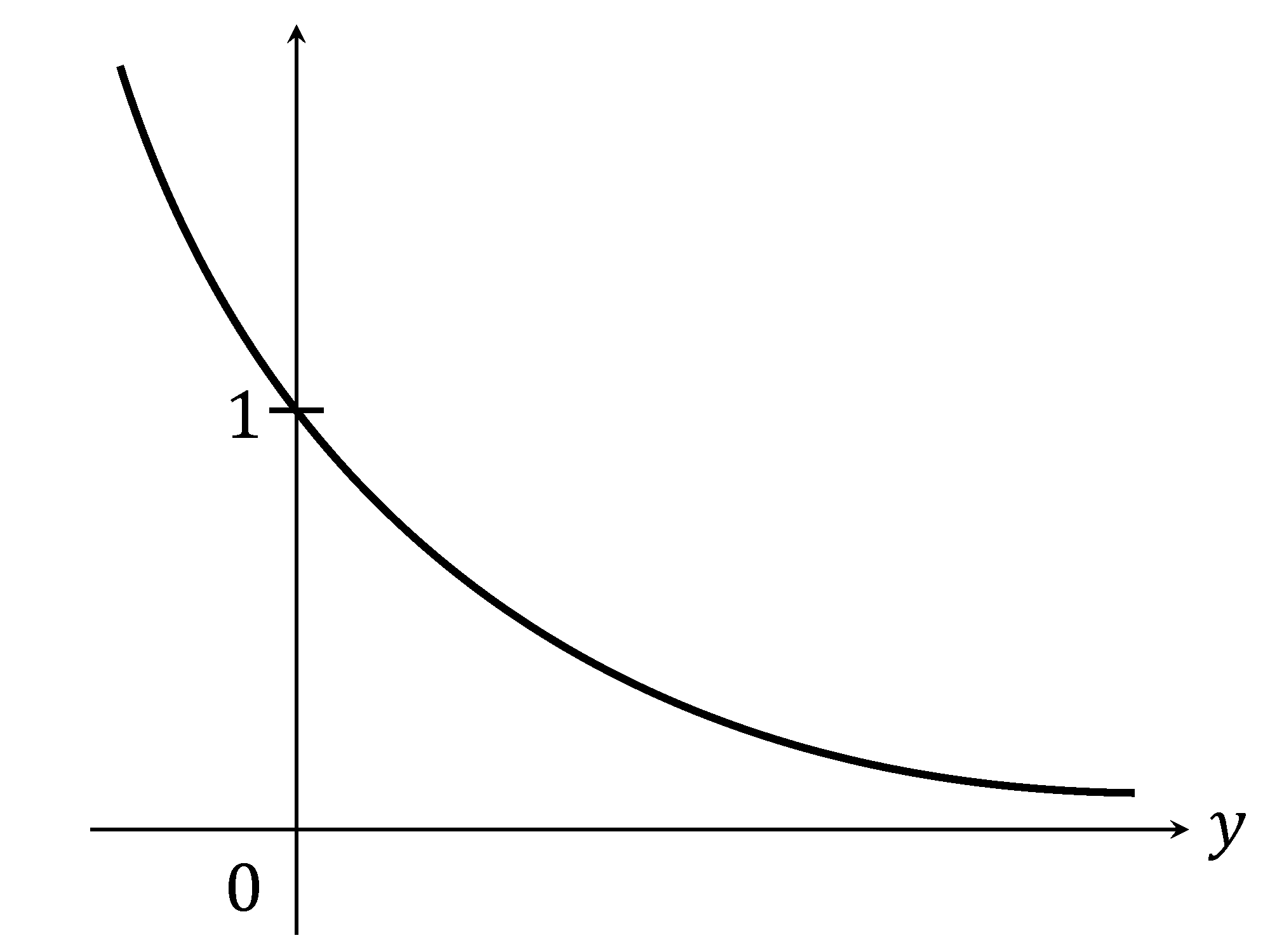}
\caption{graph of $W_1$} \label{FW1}
\end{figure}
Note that by \eqref{EIDN}, we obtain
\begin{equation} \label{EAF}
	W_1(y) = e^{-y^2/4} - \frac{y}{2} \int_y^\infty e^{-\zeta^2/4}\, d\zeta,
\end{equation}
so $W_1$ is defined for all $y\in\mathbb{R}$ and $W_1(0)=1$.
\begin{prop} \label{PWW}
$(W_1/W'_1)'(y)>0$ for $y>0$.
\end{prop}
\begin{proof} 
Since
\[
	(W_1/W'_1)' = F(y)/(W'_1)^2, \quad
	F(y) = (W'_1)^2 - W_1 W''_1,
\]
it suffices to prove that $F(y)>0$ for $y>0$.
 Differentiating $F$ to get
\begin{align*}
	F'(y) &= 2W'_1 W''_1 - W'_1 W''_1 - W_1 W'''_1 \\
	&= W'_1 W''_1 - W_1 W'''_1. \\
	&= -\frac12 \int_y^\infty e^{-\zeta^2/4}\, d\zeta\frac12 e^{-y^2/4}
	- y \int_y^\infty e^{-\zeta^2/4} \frac{d\zeta}{\zeta^2} \frac12 \left(-\frac12 \right) y e^{-y^2/4} \\
	&= \frac{e^{-y^2/4}}{4} \left\{ y^2 \int_y^\infty e^{-\zeta^2/4} \frac{1}{\zeta^2}\, d\zeta
	-\int_y^\infty e^{-\zeta^2/4}\, d\zeta \right\}.
\end{align*}
Thus
\[
	F'(y) < 0 \quad\text{for}\quad y > 0 \quad\text{since}\quad
	y^2/\zeta^2 < 1 \quad\text{for}\quad \zeta > y.
\]
We already know that $W_1(0)=1$, $W_1(\infty)=0$, where $W_1(\infty)=\lim_{y\to\infty}W_1(y)$.

Moreover,
\begin{gather*}
	W'_1(0) = -\frac12 \int_0^\infty e^{-\zeta^2/4}\, d\zeta
	= -\sqrt{\pi}/2, \quad
	W'_1(\infty) = 0 \\
	W''_1(0) = \frac12, \quad
	W''_1(\infty) = 0,
\end{gather*}
so we obtain $F(0) = \frac14\sqrt{\pi}^2-\frac12=\frac12\left(\frac{\pi}{2}-1\right)>0$ and $F(\infty)=0$.
 Thus, $F(y)>0$ for $y>0$.
 We thus conclude that $(W_1/W'_1)'(y)>0$ for $y>0$. 
\end{proof}
\begin{proof}[Proof of Lemma \ref{LUB}]
By Proposition \ref{PGe}, a solution of \eqref{EHH} is of the form
\[
	w(y) = -c_1 W_1(y) + c_0 W_0(y) + 2h_0.
\]
Since we know $W_0(y)\to\infty$ and $W_1(y)\to0$ as $y\to\infty$, $c_0$ must be zero if $w$ is bounded.
 The condition \eqref{EHD} and \eqref{EHN} are
\begin{gather*}
	c_1 W_1(\sigma) = 2h, \\
	-c_1 W'_1(\sigma) = \alpha.
\end{gather*}
Thus deleting $c_1$ to get
\begin{equation} \label{ESing}
	\frac{W_1(\sigma)}{-2W'_1(\sigma)} = \frac{h}{\alpha}.
\end{equation}
Evidently, $-W_1/(2W'_1)$ is $C^\infty$.
 By Proposition \ref{PWW}, $(-2W_1/W'_1)'<0$ on $\mathbb{R}$.
 Moreover,
\begin{align*}
	-\frac{W_1}{2W'_1}(0) &= \frac{1}{\int_0^\infty e^{-\zeta^2/4}\, d\zeta}
	= \frac{1}{\sqrt{\pi}} \\
	-\frac{W_1}{2W'_1}(\infty) &= \lim_{y\to\infty} 
	\frac{y \int_y^\infty e^{-\zeta^2/4} \frac{1}{\zeta^2}\, d\zeta}{\int_y^\infty e^{-\zeta^2/4} \, d\zeta} \\	
	&\le \lim_{y\to\infty} 
	\frac{y \int_y^\infty e^{-\zeta^2/4} \frac{1}{y^2}\, d\zeta}{\int_y^\infty e^{-\zeta^2/4}\,d\zeta} \\	
	&= \lim_{y\to\infty} \frac1y = 0.
\end{align*}
Thus, if $h/\alpha\in\left(0,1/\sqrt{\pi}\right)$ there is a unique $\sigma>0$ satisfying \eqref{ESing}.
 See Figure \ref{FWW}.
\begin{figure}[htb]
\centering
\includegraphics[width=5cm]{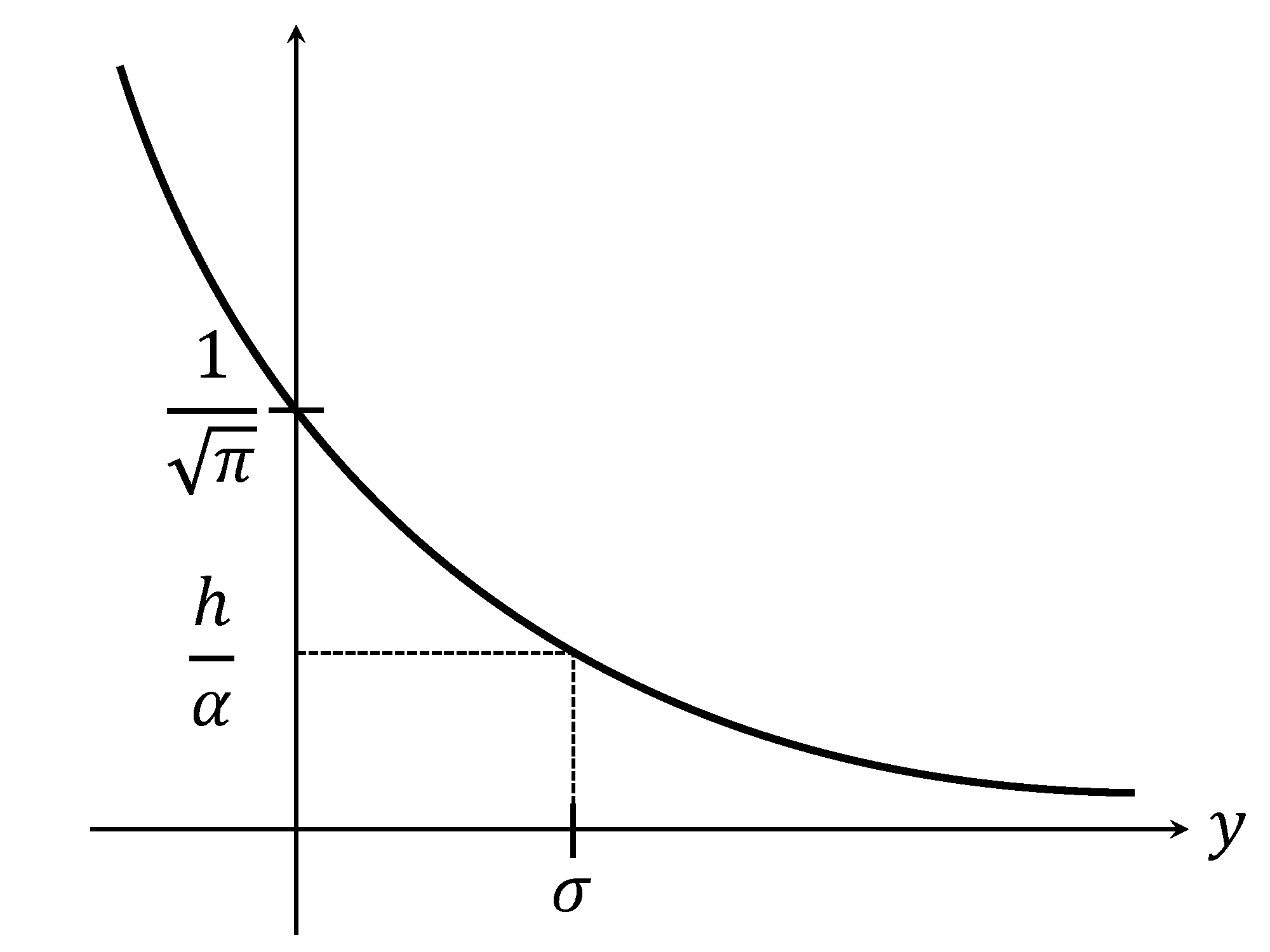}
\caption{the graph of $-W_1/(2W_1)'$} \label{FWW}
\end{figure}
We next consider the case $h/\alpha>1/\sqrt{\pi}$.
 By \eqref{EAF} and the formula for $W'_1$, we see that
\begin{equation*}
	W_1/W'_1 = \frac{e^{-y^2/4} - \frac{y}{2} \int_y^\infty e^{-y^2/4}\, d\zeta}{-\frac12\int_y^\infty e^{-\zeta^2/4}\, d\zeta}
	\to -\infty \quad\text{as}\quad y \to -\infty.
\end{equation*}
Since $W_1/W'_1$ is strictly increasing for $y<0$, we also observe that there is unique $\sigma<0$ satisfying \eqref{ESing} if $h/\alpha>1/\sqrt{\pi}$.
 If $h/\alpha=1/\sqrt{\pi}$, $\sigma$ must be zero.
 The proof for Lemma \ref{LUB} is now complete.
\end{proof}

\section*{Acknowledgements}

The work of the first author was partially supported by JSPS Grant-in-Aid for Early-Career Scientists JP22K13948. 
The work of the second author was partly supported by JSPS KAKENHI Grant Numbers JP19H00639, JP20K20342, JP24K00531 and JP24H00183 and by Arithmer Inc., Daikin Industries, Ltd.\ and Ebara Corporation through collaborative grants. 
The work of the third author was partially supported by JSPS KAKENHI Grant Number JP20K14350.

\end{document}